\documentclass[11pt,letterpaper, reqno]{amsart}

\usepackage{epsfig,amsmath,amsthm,amssymb,graphicx,hyperref, pinlabel}
\usepackage[top=1.5in, right=1.25in, left=1.25in, bottom=1.75in]{geometry}
\usepackage{caption}
\usepackage[usenames,dvipsnames,svgnames,table]{xcolor}

\usepackage[normalem]{ulem}

\newtheorem{proposition}{Proposition}[section]
\newtheorem{theorem}[proposition]{Theorem}

\newtheorem*{theorem*}{Theorem}
\newtheorem{corollary}[proposition]{Corollary}

\theoremstyle{definition}

\newtheorem{remark}[proposition]{Remark}
\numberwithin{equation}{section}

\newcommand{\tb}{\mathrm{tb}}
\newcommand{\rot}{\mathrm{rot}}

\begin{document}
\title[Topologically slice links with arbitrarily large smooth slice genus]{A family of non-split topologically slice links with arbitrarily large smooth slice genus}

\author{JungHwan Park$^\dag$}
\address{Department of Mathematics, Rice University MS-136\\
6100 Main St. P.O. Box 1892\\
Houston, TX 77251-1892}
\email{jp35@rice.edu}
\urladdr{http://math.rice.edu/$\sim$jp35}
\author{Arunima Ray$^{\dag\dag}$}
\address{Department of Mathematics MS-050, Brandeis University, 415 South St., Waltham, MA 02453}
\email{aruray@brandeis.edu}
\urladdr{http://people.brandeis.edu/$\sim$aruray}
\date{\today}
\thanks{$^\dag$ Partially supported by National Science Foundation grant DMS-1309081.}
\thanks{$^{\dag\dag}$ Partially supported by an AMS--Simons Travel Grant.}


\begin{abstract}
We construct an infinite family of topologically slice 2--component boundary links $\ell_i$, none of which is smoothly concordant to a split link, such that $g_4(\ell_i)=i$.
\end{abstract}

\maketitle

%=============================================================
\section{Introduction}\label{Introduction}
A $k$--component link $L$ is the isotopy class of an embedding $\bigsqcup_k S^1\rightarrow S^3$ and a knot is simply a $1$--component link. A link is said to be smoothly slice if its components bound a disjoint collection of smoothly embedded disks in $B^4$; if there exists such a disjoint collection of merely locally flat disks we say that the link is topologically slice. The study of smoothly and topologically slice links is closely connected with the study of smooth and topological $4$--manifolds; e.g.\ any knot which is topologically slice but not smoothly slice~\cite{En95,Gom86, HeK12, HeLivRu12, Hom14}) gives rise to an exotic copy of $\mathbb{R}^4$~\cite[Exercise 9.4.23]{GS99}.

In an approach to approximating sliceness of links, we may consider surfaces bounded by a link in $B^4$. The minimal genus of a smooth embedded connected oriented surface in $B^4$ with boundary a given link $L$ is said to be the smooth slice genus of $L$, whereas the minimal genus of such a locally flat surface is called the topological slice genus of $L$. We denote these by $g_4(L)$ and ${g_4}^{top}(L)$ respectively. Note that if a link is smoothly (resp.\ topologically) slice it has zero smooth (resp.\ topological) slice genus. The converse is not true; e.g.\ the Hopf link (with either orientation) has smooth and topological slice genus zero, but is neither smoothly nor topologically slice. (Since slice surfaces must be oriented, the slice genus of a link depends on the relative orientation of the link components in general.) It is easy to see that the smooth (resp.\ topological) slice genus is an invariant of smooth (resp.\ topological) concordance of links. 

For any link $L$ we see that  ${g_4}^{top}(L) \leq g_4(L)$, since any smooth embedding of a surface is locally flat. Understanding the extent to which these two quantities are different can be seen as refining the question of when topologically slice knots may be smoothly non-slice. In particular, we focus on the following natural questions. 
\begin{itemize}
\item Are there examples of links which satisfy ${g_4}^{top}(L) < g_4(L)$?
\item Can the difference between $g_4(\cdot)$ and ${g_4}^{top}(\cdot)$ be arbitrarily large?
\end{itemize}
The above have been studied extensively for knots (see \cite{Don83, CG88, Tan98, FM15}). Here we will focus on 2--component links, for which we show that the answer to both questions is yes. 

\begin{theorem}\label{thm_main} For any integer $i \geq 0$, there exists a 2--component link $\ell_i$ such that\begin{enumerate}
\item $g_4(\ell_i)=i$ (consequently, the links $\ell_i$ are distinct in smooth concordance),
\item $\ell_i$ is not smoothly concordant to a split link,
\item $\ell_i$ is a boundary link,
\item $\ell_i$ is topologically slice (in particular, ${g_4}^{top}(\ell_i)=0$).
\end{enumerate}
\end{theorem}

Removing condition (2) makes the theorem trivial, since we can use the links $\ell_i=K_i\sqcup U$, where each $K_i$ is a topologically slice knot with $g_4(K_i)=i$, $U$ is the unknot, and $\sqcup$ indicates taking a split union. Moreover, examples satisfying (2-4) are already known by~\cite[Theorem B]{RS13}. We will show that our examples are distinct from those in smooth concordance in Proposition~\ref{prop:smoothly_distinct}.

\subsection*{Acknowledgements}The first author would like to thank his advisor Shelly Harvey for her guidance and helpful discussions. The second author also thinks Shelly is pretty cool. 

We are indebted to the anonymous referee for comments that led to substantially improved exposition. 

%=============================================================
\section{Preliminaries}\label{sec:preliminaries}

%\section{Legendrian knots and the slice--Bennequin inequality}

This section consists of a brief overview of Legendrian knots, limited to the material we need for our proof. For more precise definitions and details, we direct the reader to~\cite{Etn05}. 

Recall that the standard contact structure on $\mathbb{R}^3$ is given by the kernel of the $1$--form $dz - y\, dx$. Then the standard contact structure on $S^3$ is defined such that if one removes a single point from $S^3$ the resulting contact structure is contactomorphic to the standard contact structure on $\mathbb{R}^3$. An embedding $\mathcal{K}$ of a knot $K$ in $S^3$ is Legendrian if $\mathcal{K}$ is tangent to the 2--planes of the standard contact structure on $S^3$. Legendrian knots may be studied concretely using their front projections, i.e.\ since a knot is compact we may consider it to be in $\mathbb{R}^3\subseteq S^3$ and then use the projection onto the $xz$--plane. The middle and right panel of Figure~\ref{fig:leg_satellite} show front projections of two Legendrian knots. There are two classical invariants for Legendrian knots, the Thurston--Bennequin number, $\tb(\cdot)$, and the rotation number, $\rot(\cdot)$. Given a front projection $\Pi(\mathcal{K})$ of a Legendrian knot $\mathcal{K}$, we have the following formulae:
\begin{align}
\tb(\mathcal{K})&=\text{writhe}(\Pi(\mathcal{K}))-\frac{1}{2}\#\text{cusps}(\Pi(\mathcal{K}))\\
\rot(\mathcal{K})&=\frac{1}{2}\#\text{downward-moving cusps}(\Pi(\mathcal{K}))-\frac{1}{2}\#\text{upward-moving cusps}(\Pi(\mathcal{K}))
\end{align}

Our main tool in this paper is the slice--Bennequin inequality (see \cite{Rud95, Rud97, Etn05, AM97, LM98}), which says that for any Legendrian representative $\mathcal{K}$ of a knot $K$, $$\tb(\mathcal{K}) + |\rot(\mathcal{K})| \leq  2\tau(K) - 1 \leq 2g_4(K) -1$$
where $\tau(\cdot)$ is Ozv\'{a}th--Szab\'{o}'s concordance invariant from Heegaard--Floer homology~\cite{OzSz04}, and the first inequality is from~\cite{Plam04}. Recall that $\tau$ is additive under connected sum and insensitive to the orientation of a knot.

The standard contact structure on $S^1\times \mathbb{R}^2$ is also defined as the kernel of the $1$--form $dz-y\,dx$, where we identify $S^1\times \mathbb{R}^2$ with $\mathbb{R}^3$ modulo $(x, y, z) \sim (x+1, y, z)$. As before an embedding $\mathcal{P}$ of a knot $P$ in $S^1 \times \mathbb{R}^2$ (called a pattern) is Legendrian if $\mathcal{P}$ is  tangent to the $2$-planes of the standard contact structure on $S^1 \times \mathbb{R}^2$. As in $\mathbb{R}^3$, we have front projections on the $xz$--plane, where the $x$--direction is understood to be periodic. We will draw these front projections in $[0,1]\times\mathbb{R}^2$ as shown in the left panel of Figure~\ref{fig:leg_satellite}, where the dashed lines indicate that the boundary should be identified. Using such front projections, we compute the Thurston--Bennequin number and rotation number of Legendrian patterns using the same combinatorical formulae as for knots given above. The winding number, $w(\cdot)$, of a Legendrian pattern is the signed number of times it wraps around the longitude of $S^1\times \mathbb{R}^2$. 

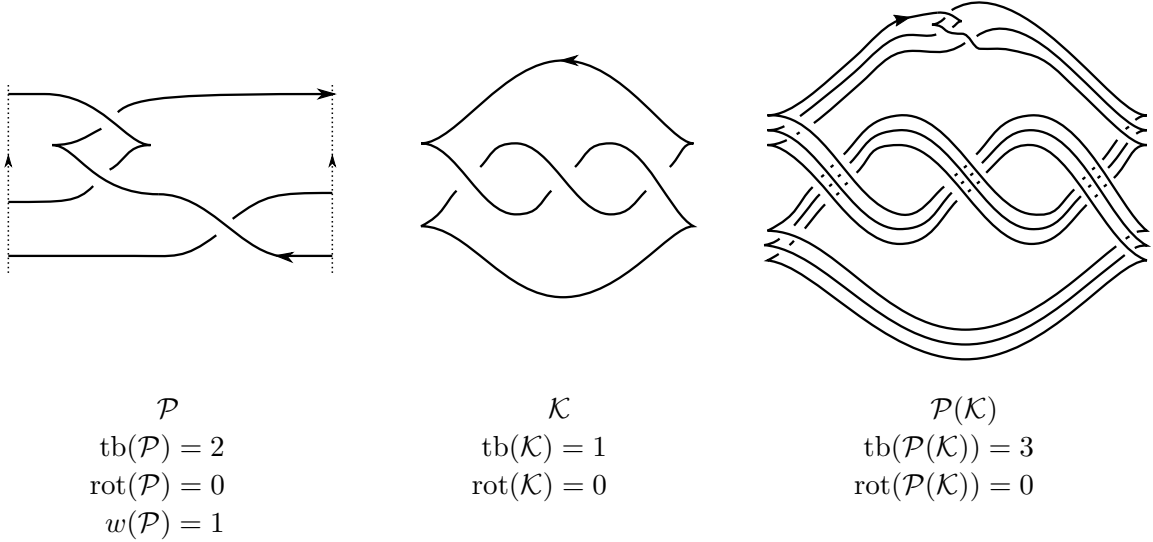
\begin{figure}[t]
  \includegraphics[width=\textwidth]{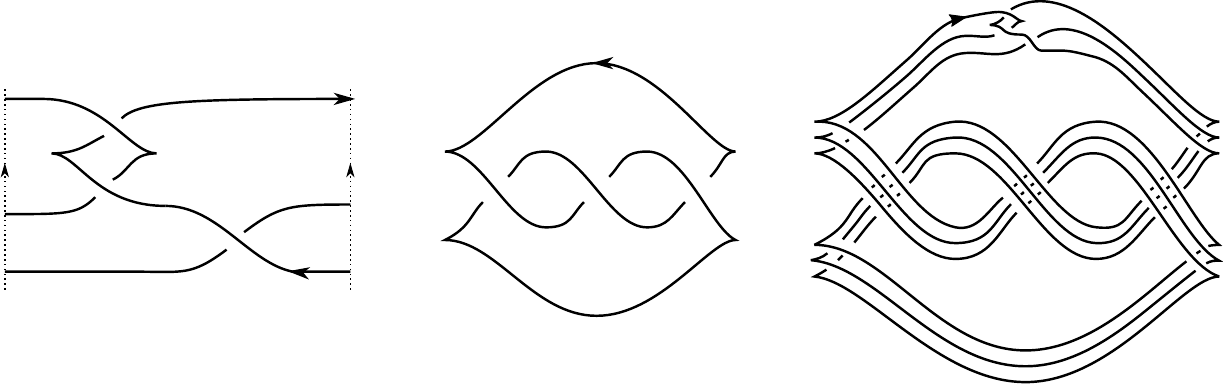}
  \put(-5.2,-0.3){$\mathcal{P}$}
  \put(-3.15,-0.3){$\mathcal{K}$}
  \put(-1.15,-0.3){$\mathcal{P}(\mathcal{K})$}
  \put(-5.5,-0.5){$\tb(\mathcal{P})=2$}
  \put(-5.55, -0.7){$\rot(\mathcal{P})=0$}
  \put(-5.47, -0.9){$w(\mathcal{P})=1$}
  \put(-3.5, -0.5){$\tb(\mathcal{K})=1$}
  \put(-3.55, -0.7){$\rot(\mathcal{K})=0$}
  \put(-1.5,-0.5){$\tb(\mathcal{P}(\mathcal{K}))=3$}
  \put(-1.55, -0.7){$\rot(\mathcal{P}(\mathcal{K}))=0$}
\caption{The Legendrian satellite operation}\label{fig:leg_satellite}
\end{figure}

Let $\mathcal{P}$ be a Legendrian pattern in $S^1 \times \mathbb{R}^2$ with $n$ end points, and $\mathcal{K}$ be a Legendrian knot. Then the Legendrian satellite operation yields a Legendrian knot $\mathcal{P}(\mathcal{K})$ by taking $n$ vertical parallel copies of $K$ and inserting $\mathcal{P}$ in an appropriately oriented strand of $\mathcal{K}$ (see Figure~\ref{fig:leg_satellite} for an example). It is easy to see that $\mathcal{P}(\mathcal{K})$ is a Legendrian diagram for the $\tb(\mathcal{K})$--twisted satellite of $K$. (For a detailed discussion of the Legendrian satellite operation see \cite{Ng01, NgTray04, Ray15}.) Hence when $\tb(\mathcal{K})=0$, $\mathcal{P}(\mathcal{K})$ represents the classical untwisted satellite with pattern $P$ and companion $K$ (see Figure~\ref{fig:leg_whitehead}). The following proposition establishes the relationship between the Thurston--Bennequin numbers and rotation numbers of a Legendrian pattern, a Legendrian knot, and the associated Legendrian satellite.

\begin{proposition}[Remark 2.4 of \cite{Ng01}]\label{prop:legformulas} For a Legendrian pattern $\mathcal{P}$ and a Legendrian knot $\mathcal{K}$,
$$\tb(\mathcal{P}(\mathcal{K}))=w(\mathcal{P})^2\tb(\mathcal{K})+\tb(\mathcal{P})$$ 
$$\rot(\mathcal{P}(\mathcal{K}))=w(\mathcal{P})\rot(\mathcal{K})+\rot(\mathcal{P}).$$
\end{proposition}

%=============================================================
\section{Proof of main theorem}\label{sec:maintheorem}

For this section, we fix a Legendrian diagram $\mathcal{K}$ of a knot $K$ with the following properties:
\begin{enumerate}
\item $K$ is topologically slice.
\item $g_3(K) = g_4(K) = \tau(K) =1$.
\item $\tb(\mathcal{K}) = 0$.
\item $\rot(\mathcal{K}) = 2g_4(K) -1 = 1$.
\end{enumerate}

Examples of such knots can be easily found, as follows. Let $J$ be any knot with a Legendrian realization $\mathcal{J}$ satisfying $\tb(\mathcal{J})=0$ and $\tau(J)>0$, e.g.\ the right-handed trefoil. Any knot with positive maximal Thurston--Bennequin number has positive $\tau$ and such a Legendrian realization. Now perform the Legendrian satellite operation on $\mathcal{J}$ using the pattern for untwisted positive Whitehead doubling shown in Figure~\ref{fig:leg_whitehead}. We call the resulting Legendrian knot $\mathcal{K}$, which is a realization of the topological knot type $K$ (note that $K$ is the positive untwisted Whitehead double of $J$). We know that $K$ is topologically slice since it has Alexander polynomial one~\cite{Freed82}. Using Proposition~\ref{prop:legformulas}, we see that $\tb(\mathcal{K})=0$ and $\rot(\mathcal{K})=1$, and by~\cite{He07}, we see that $g_3(K)=g_4(K)=\tau(K)=1$.

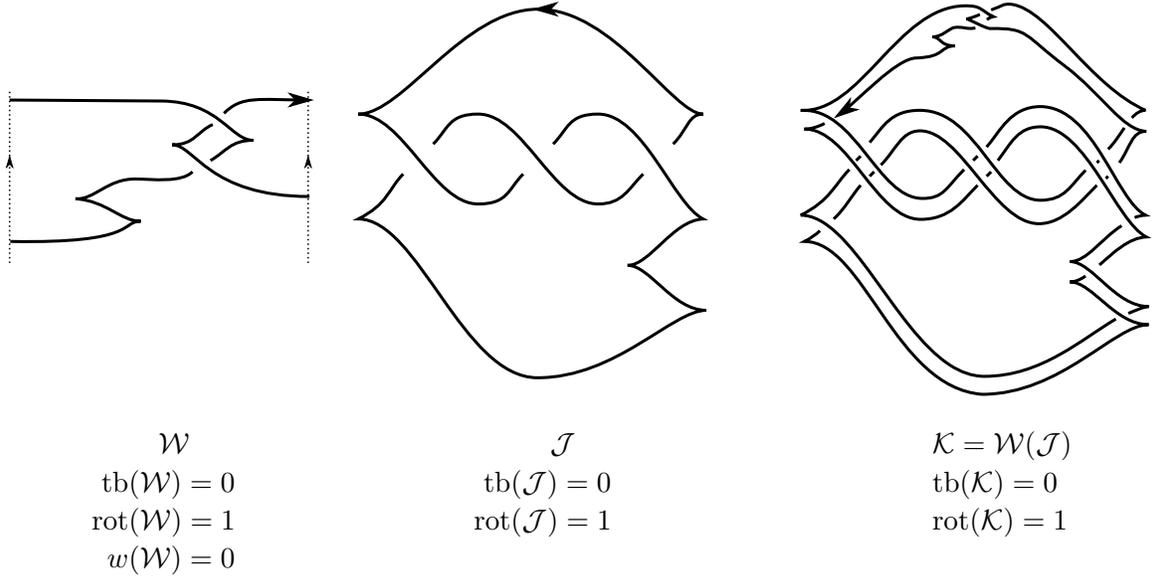
\begin{figure}[b]
  \includegraphics[width=\textwidth]{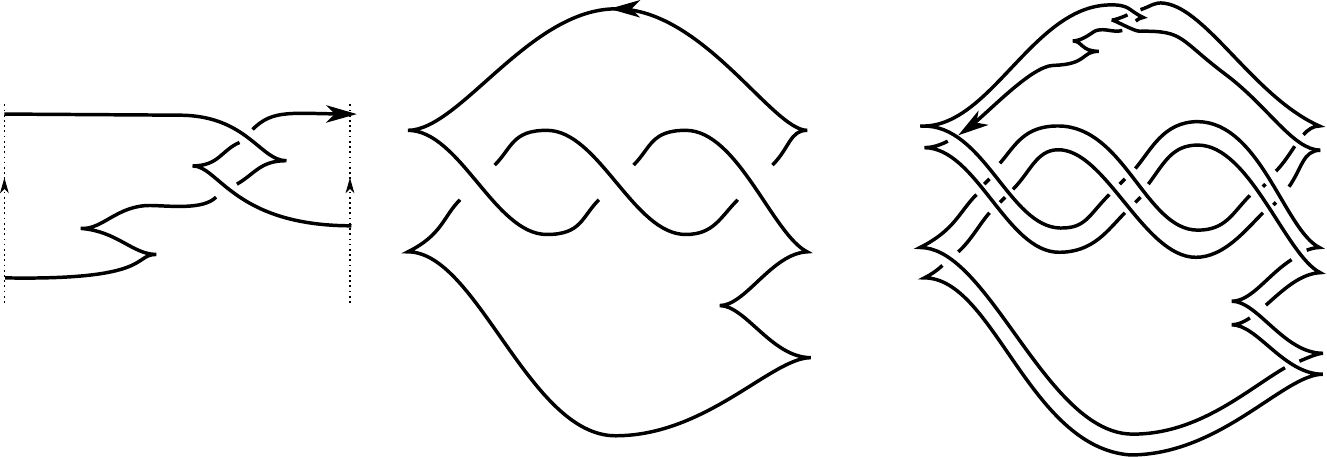}
  \put(-5.2,-0.3){$\mathcal{W}$}
  \put(-3.15,-0.3){$\mathcal{J}$}
  \put(-1.15,-0.3){$\mathcal{K}=\mathcal{W}(\mathcal{J})$}
  \put(-5.5,-0.5){$\tb(\mathcal{W})=0$}
  \put(-5.55, -0.7){$\rot(\mathcal{W})=1$}
  \put(-5.47, -0.9){$w(\mathcal{W})=0$}
  \put(-3.5, -0.5){$\tb(\mathcal{J})=0$}
  \put(-3.55, -0.7){$\rot(\mathcal{J})=1$}
  \put(-1.15,-0.5){$\tb(\mathcal{K})=0$}
  \put(-1.15, -0.7){$\rot(\mathcal{K})=1$}
\caption{Constructing the knots $\mathcal{K}$.}\label{fig:leg_whitehead}
\end{figure}

Since $\tb(\mathcal{K})=0$, from Section~\ref{sec:preliminaries}, we know that for any Legendrian diagram $\mathcal{P}$ for a pattern $P$, the Legendrian satellite $\mathcal{P}(\mathcal{K})$ is a Legendrian diagram for the untwisted satellite $P(K)$. 

\begin{figure}[b]
\centering
\includegraphics[width=5in]{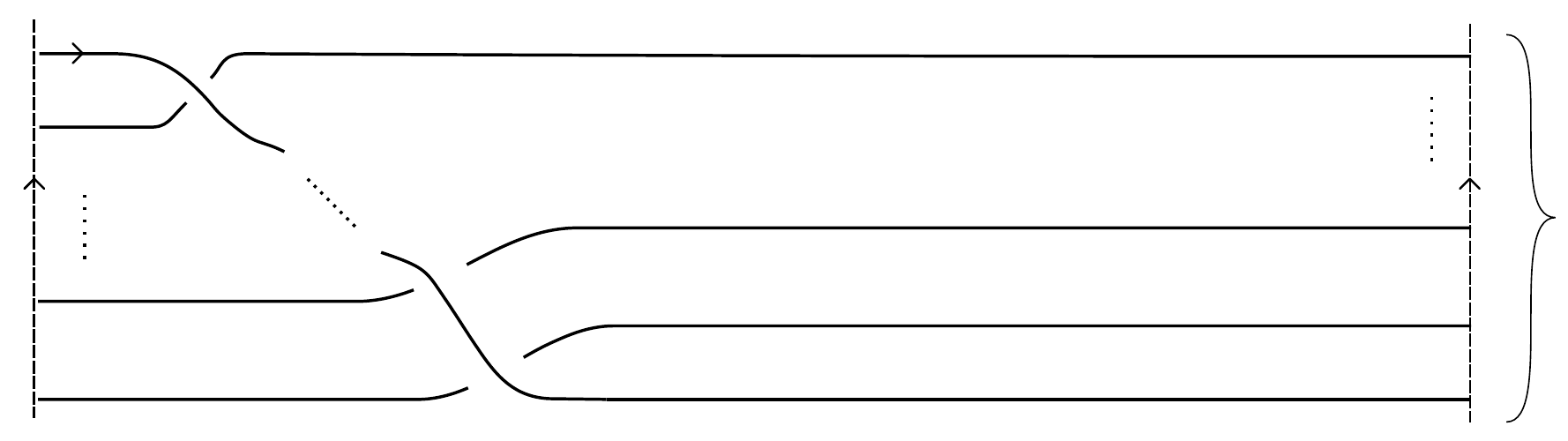}
\put(0,0.67){$i$ strands}
\caption{A Legendrian diagram $\mathcal{P}_i$ for the pattern $P_i$. We compute that $\tb(\mathcal{P}_i)=i-1, \rot(\mathcal{P}_i)=0$ and $w(\mathcal{P}_i)=i$.}\label{fig:Pi}
\end{figure}

We start with a few propositions. For any positive integer $i$, consider the Legendrian diagram $\mathcal{P}_i$ for a pattern $P_i$, given in Figure~\ref{fig:Pi}. Notice that the satellite knot $P_i(K)$ is the $(i,1)$ cable of $K$.

\begin{proposition}\label{prop:Pi}For the pattern $P_i$ and any integer $i \geq 1$, we have 
$$g_4(P_i(K))= \tau(P_i(K)) = i.$$
\end{proposition}
\begin{proof} Using Proposition~\ref{prop:legformulas}, we calculate:
$$\tb(\mathcal{P}_i(K))=w(\mathcal{P}_i)^2\tb(\mathcal{K})+\tb(\mathcal{P}_i)=i^2\cdot 0 + (i-1) = i - 1$$ 
$$\rot(\mathcal{P}_i(K))=w(\mathcal{P}_i)\rot(\mathcal{K})+\rot(\mathcal{P}_i)=i\cdot 1 + 0 = i.$$
Then by the slice--Bennequin inequality we have the following:
$$(i-1) + \lvert i\rvert = 2i -1 \leq 2\tau(P_i(K)) - 1 \leq 2g_4(P_i(K)) -1$$
and thus,
$$i \leq \tau(P_i(K)) \leq g_4(P_i(K)).$$
Note that we can change $P_i(K)$ into the $(i,0)$ cable of $K$ by performing $i-1$ band sums. Since $g_4(K)=1$ there is a surface $\Sigma$ in $B^4$ with $g(\Sigma)=1$ and $\partial \Sigma=K$, and we can take $i$ parallel copies of $\Sigma$ to get a genus $i$ surface smoothly embedded in $B^4$ bounded by $P_i(K)$. This shows that $g_4(P_i(K)) \leq i$. Combining this with the above, we conclude that $g_4(P_i(K))= \tau(P_i(K)) = i$.
\end{proof}

Note that we can also see that $\tau(P_i(K))=i$ by using Hom's formula from~\cite{Hom14}, since $P_i(K)$ is the $(i,1)$ cable of $K$ and, by~\cite{Hom14}, $\varepsilon(K)=1$. 

For any positive integer $i$, consider the Legendrian diagram $\mathcal{Q}_i$ for a pattern $Q_i$, shown in Figure~\ref{fig:Qi}. This pattern is similar to the one shown in \cite[Figure 9]{Ray15}, but $w(Q_i)=0$ whereas the pattern from \cite{Ray15} has winding number one. 

\begin{figure}[t]
\centering
\includegraphics[width=4in]{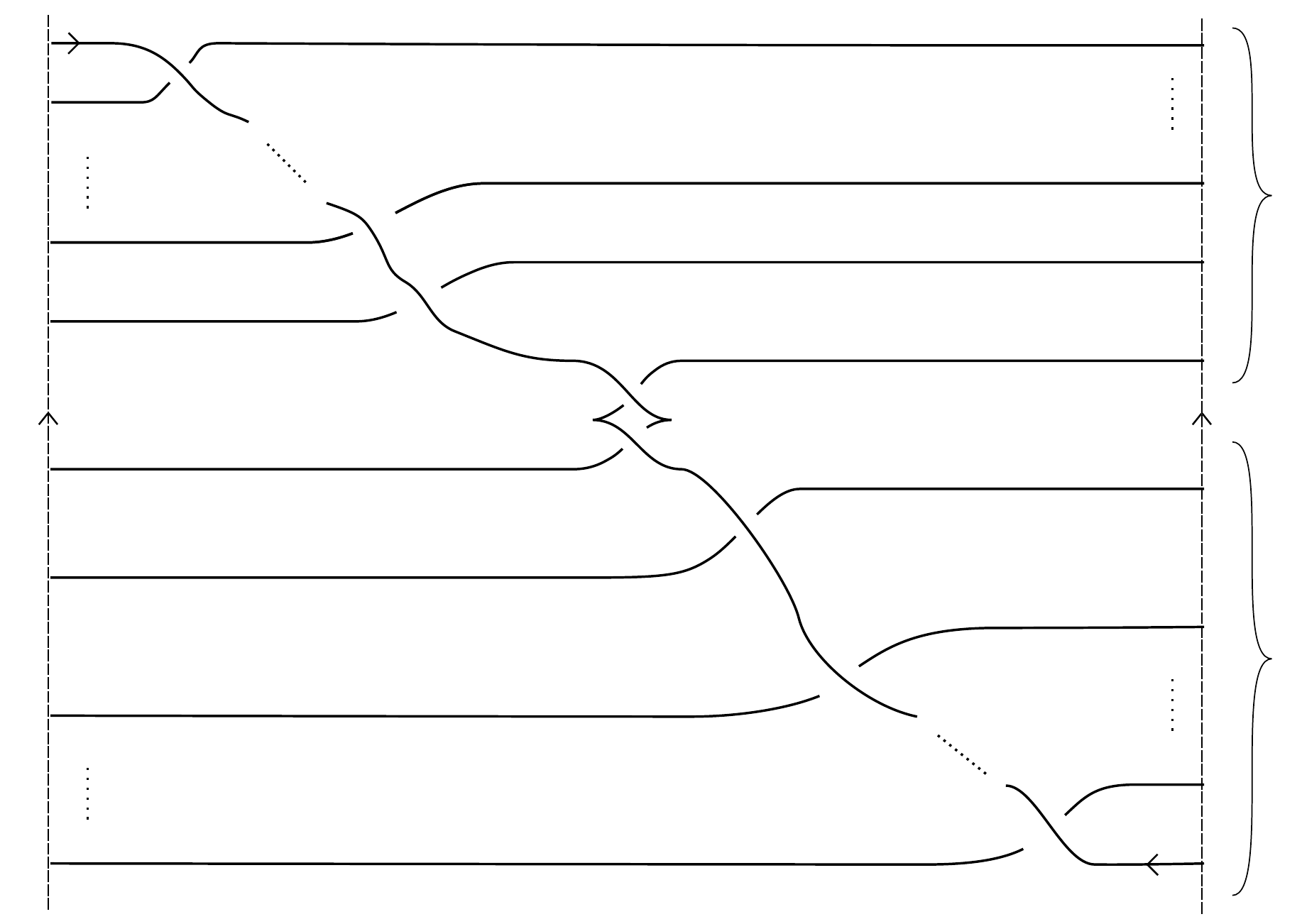}
\put(0,0.75){$i$ strands}
\put(0,2.2){$i$ strands}
\caption{A Legendrian diagram $\mathcal{Q}_i$ for the pattern $Q_i$. We compute that $\tb(\mathcal{Q}_i)=2i-1, \rot(\mathcal{Q}_i)=0$ and $w(\mathcal{Q}_i)=0$.}\label{fig:Qi}
\end{figure}

\begin{proposition}\label{prop:Qi} For the pattern $Q_i$ and any integer $i \geq 1$, we have 
$$g_4(Q_i(K))= \tau(Q_i(K)) = i.$$
\end{proposition}
\begin{proof} Using Proposition~\ref{prop:legformulas}, we calculate:
$$\tb(\mathcal{Q}_i(K))=w(\mathcal{Q}_i)^2\tb(\mathcal{K})+\tb(\mathcal{Q}_i)=0^2\cdot 0 + (2i-1) = 2i - 1$$ 
$$\rot(\mathcal{Q}_i(K))=w(\mathcal{Q}_i)\rot(\mathcal{K})+\rot(\mathcal{Q}_i)=0\cdot 1 + 0 = 0.$$
Then by the slice--Bennequin inequality we have the following:
$$(2i-1) + |0| = 2i -1 \leq 2\tau(Q_i(K)) - 1 \leq 2g_4(Q_i(K)) -1$$
and thus,
\begin{equation}\label{eqn:inequality}
i \leq \tau(Q_i(K)) \leq g_4(Q_i(K)).
\end{equation}

Notice that $Q_1(K)$ is just the positive clasped Whitehead double of $K$ and thus $g_4(Q_1(K)) \leq g_3(Q_1(K)) = 1$. By~(\ref{eqn:inequality}), $1\leq g_4(Q_1(K))$ and thus, $g_4(Q_1(K))=1$. Additionally, there exists a genus one cobordism between $Q_i(K)$ and $Q_{i+1}(K)$ for $i\geq 1$, shown in Figure~\ref{fig:cobordism}, obtained by changing a crossing at the clasp in $Q_{i+1}(K)$. By induction, we see that $g_4(Q_i(K)) \leq i$, and combining this with~\ref{eqn:inequality}, we see that $g_4(Q_i(K))= \tau(Q_i(K)) = i$.
\end{proof}

\begin{figure}[h!]
\centering
\includegraphics[width=6in]{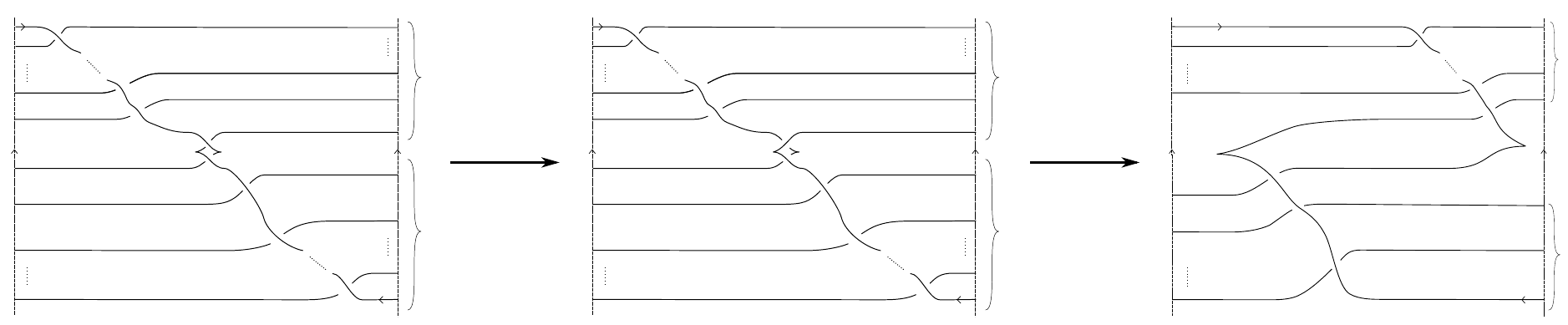}
\put(-5.4,-0.2){ $\mathcal{Q}_{i+1}$}
\put(-3.1,-0.2){ $Q_{i}$}
\put(-0.9,-0.2){ $\mathcal{Q}_{i}$}
\put(-4.4,0.88){ $i+1$}
\put(-4.4,0.3){ $i+1$}
\put(-2.2,0.88){ $i+1$}
\put(-2.2,0.3){ $i+1$}
\put(-0.05,0.94){ $i$}
\put(-0.05,0.22){ $i$}
\caption{A genus one cobordism from $Q_{i+1}$ to $Q_{i}$. Since the cobordism shown occurs in $S^1\times D^2$, this also shows a cobordism from $Q_{i+1}(K)$ to $Q_i(K)$. The first arrow is obtained by changing a crossing at the clasp. Notice that the second diagram is no longer Legendrian. The second arrow is obtained by an isotopy and results in the familiar diagram $\mathcal{Q}_i$.}\label{fig:cobordism}
\end{figure}

\begin{figure}[t]
\centering
\includegraphics[width=3in]{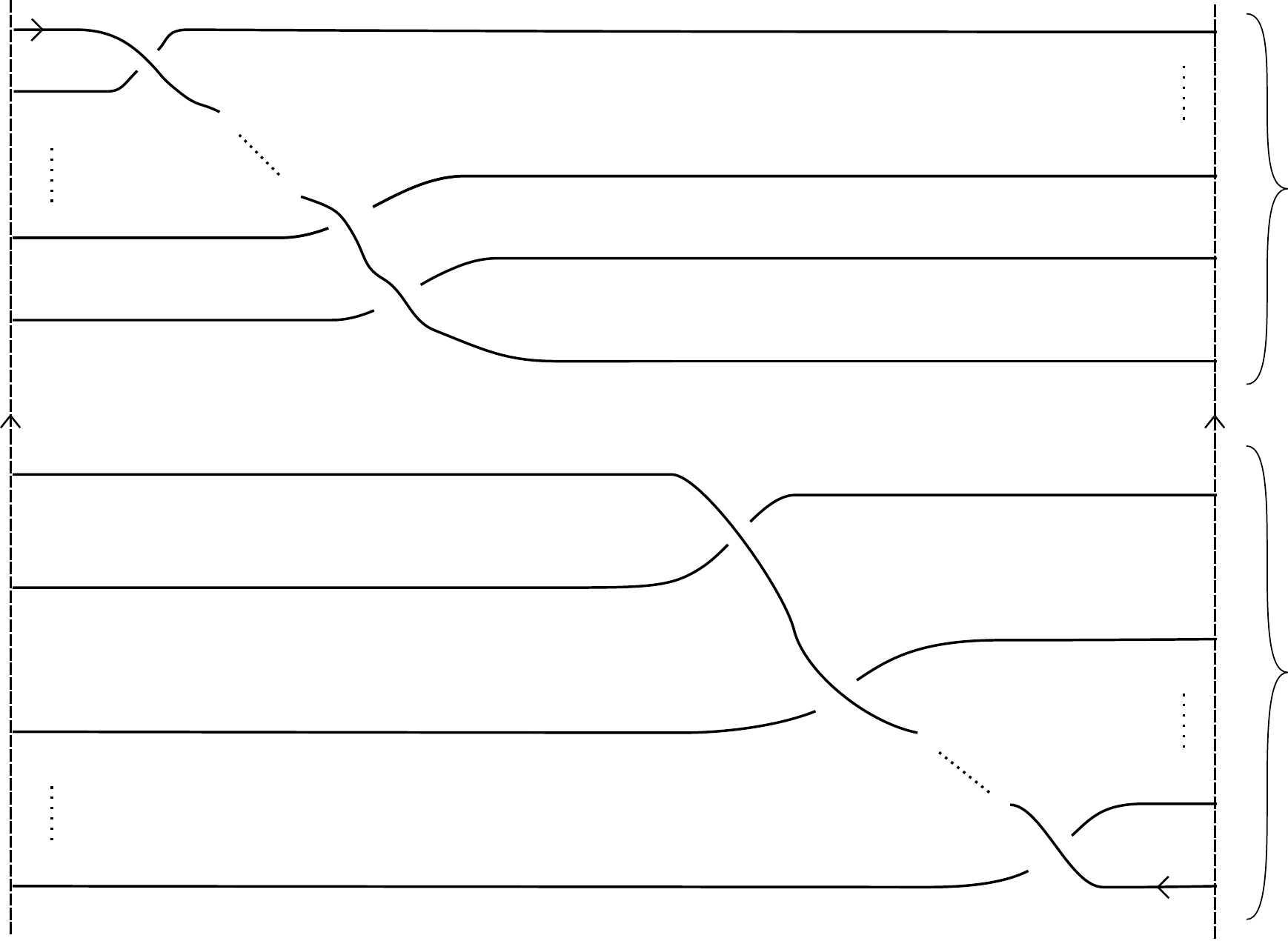}
\put(0,0.6){ $i+1$ strands}
\put(0,1.725){ $i+1$ strands}
\caption{A Legendrian diagram $\mathcal{L}_i$ for the pattern $L_i$. We compute that $\tb(\mathcal{L}_i)=2i, \rot(\mathcal{L}_i)=0$ and $w(\mathcal{L}_i)=0$.}\label{fig:Li}
\end{figure}

We are now ready to prove the main theorem, which we restate below. For each positive integer $i$, consider the pattern $L_i$ shown in Figure~\ref{fig:Li}. Notice that the link $L_i(K)$, if we ignore the orientation of the strands, is obtained by performing the $(i+1,1)$ cabling operation on each component of the $(2,0)$ cable of $K$. 

\newtheorem*{thm_main}{Theorem~\ref{thm_main}}
\begin{thm_main} For any integer $i\geq 0$, there exists a $2$--component link $\ell_i$ such that
\begin{enumerate}
\item $g_4(\ell_i)=i$ (consequently, the links $\ell_i$ are distinct in smooth concordance),
\item $\ell_i$ is not smoothly concordant to a split link.
\item $\ell_i$ is a boundary link.
\item $\ell_i$ is topologically slice (in particular, ${g_4}^{top}(\ell_i)=0$.)
\end{enumerate}
\end{thm_main}

\begin{proof} For any integer $i\geq 0$, let $\ell_i$ denote the 2--component link $L_i(K)$. We first show $g_4(L_i(K))=i$. When $i=0$, if we disregard orientation, $L_0(K)$ is simply the $(2,0)$ cable of $K$. Since the components of $L_0(K)$ has opposite orientation, they cobound an annulus which implies that $g_4(L_0(K))=0$. For $i \geq 1$, notice that there is a cobordism from $Q_{i+1}(K)$ to $L_{i}(K)$ and a cobordism from $L_{i}(K)$ to $Q_{i}(K)$ (see Figure~\ref{fig:cobordism2}). By the first cobordism and Proposition~\ref{prop:Qi}, we have $i+1=g_4(Q_{i+1}(K)) \leq g_4(L_{i}(K))+1$ and by the second cobordism and Proposition~\ref{prop:Qi}, we have $g_4(L_{i}(K)) \leq g_4(Q_{i}(K))=i$. Hence we can conclude $g_4(L_i(K))=i$. 

\begin{figure}[t]
\centering
\includegraphics[width=6in]{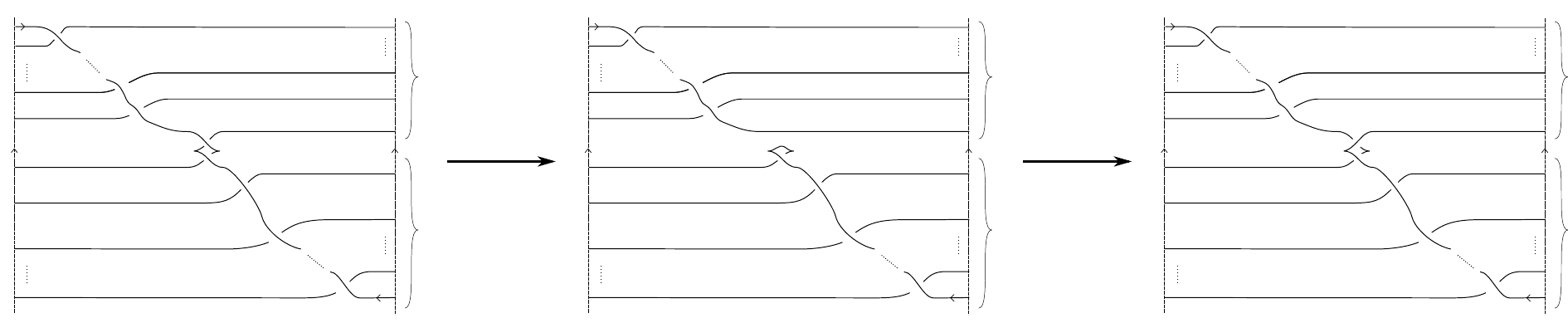}
\put(-5.4,-0.2){ $\mathcal{Q}_{i+1}$}
\put(-3.1,-0.2){ $L_{i}$}
\put(-0.9,-0.2){ $Q_{i}$}
\put(-4.4,0.88){ $i+1$}
\put(-4.4,0.3){ $i+1$}
\put(-2.2,0.88){ $i+1$}
\put(-2.2,0.3){ $i+1$}
\put(-0,0.88){ $i+1$}
\put(-0,0.3){ $i+1$}
\caption{The first arrow indicates a cobordism between $\mathcal{Q}_{i+1}(K)$ and $L_{i}(K)$ and the second arrow indicates a cobordism between $L_{i}(K)$ and $Q_{i}(K)$. Note that the right panel is the middle panel of Figure~\ref{fig:cobordism}}\label{fig:cobordism2}
\end{figure}
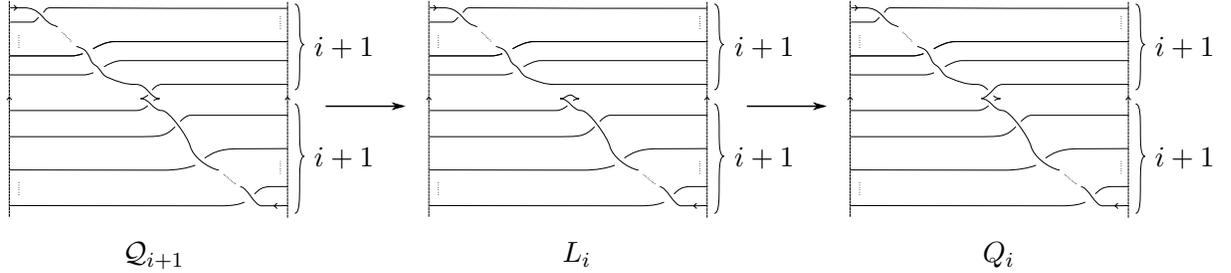

%It was shown in~\cite[Theorem A]{RS13} that if the $(2,0)$ cable of $K$ is smoothly concordant to a split link then $\tau(K)=0$. Thus, since we chose $K$ to have $\tau(K)=1$, $L_0(K)$ is not smoothly concordant to a split link. 
For $i \geq 0$, assume that $L_i(K)$ is smoothly concordant to a split link. Then it was observed in \cite[Lemma $2.1$]{RS13} that $L_i(K)$ is smoothly concordant to $K_{(i+1,1)} \sqcup r(K_{i+1,1})$ where $K_{i+1,1}$ is the $(i+1,1)$ cable of $K$, $r(K_{i+1,1})$ is $K_{i+1,1}$ with reversed orientation, and $\sqcup$ indicates a split union. Using this observation, we see that $g_4(K_{i+1,1} \sqcup r(K_{i+1,1})) = g_4(L_i(K))= i$ and thus, $g_4(K_{i+1,1} \# r(K_{i+1,1}))=i$ (see \cite[Proposition $3.3$]{CH14}). This is a contradiction since, $\tau(K_{i+1,1} \# r(K_{i+1,1})) = \tau(K_{i+1,1}) + \tau(r(K_{i+1,1})) = 2\tau(K_{i+1,1}) = 2\tau(P_{i+1}(K))$ and by Proposition~\ref{prop:Pi}, $\tau(P_{i+1}(K))= i+1$.

It is straightforward to see that $L_i(K)$ is a boundary link by construction: use parallel copies of a Seifert surface for $K$. Lastly $L_i(K)$ is topologically slice since $K$ is topologically slice.
\end{proof}

\begin{proposition}\label{prop:smoothly_distinct}
The examples $\ell_i$ from Theorem~\ref{thm_main} are distinct in smooth concordance from the examples given in~\cite[Theorem B]{RS13}.
\end{proposition}
\begin{proof}
The examples in~\cite[Theorem B]{RS13} consist of the $(2,0)$ cables, with either the parallel or antiparallel orientation, of a family of knots $\{Wh(J_i)\}$, where $J_i$ is either the connected sum of $i$ copies of the right-handed trefoil, or the torus knot $T_{2, 2i+1}$. It is easy to see from~\cite[Corollary 3.2]{RS13} that their argument also applies for $(2,0)$ cables of the connected sum of $i$ copies of the Whitehead double of the right-handed trefoil knot. We will show that our examples are distinct from these cables in smooth concordance. Since the Ruberman--Strle examples are (2,0) cables, we may choose the antiparallel orientation of the two strands; with this orientation, the smooth slice genus of the link is zero. For our examples, we saw in Theorem~\ref{thm_main}, that $g_4(\ell_i)=i$. Let $\ell'_i$ denote the link where we switch the orientation of one component. Then we may attach a single band to see a genus zero cobordism between $\ell'_i$ and $P_{2i+2}(K)$ (or its reverse). Then by Proposition~\ref{prop:Pi}, $g_4(\ell'_i)\geq 2i+1$. On the other hand, if the link $\ell_i$ were concordant to a (2,0) cable with some orientation, either $\ell_i$ or $\ell'_i$ would have zero slice genus.  

In~\cite{RS13}, we also see some examples due to Livingston consisting of Bing doubles of certain topologically slice knots. As before, we can choose an orientation for the Bing double such that there is a genus zero cobordism to the untwisted Whitehead double, and thus the slice genus of the link with this orientation is at most one. By our previous argument, our links $\ell_i$ are distinct in concordance from Livingston's examples as long as $i\geq 2$.
\end{proof}

Note that above we have shown that the difference between the smooth slice genus of 2--component topologically slice links with the two different relative orientations for the strands can be arbitrarily large. This is also true for the examples given in~\cite{RS13}. 

In~\cite{Cav15}, Cavallo introduced a generalization of Ozv\'{a}th--Szab\'{o}'s concordance invariant $\tau$ for links. He established the following inequality (see~\cite[Propositions 1.4 and 1.5]{Cav15}):
$$ \tb(\mathcal{L}) + \lvert \rot(\mathcal{L})\rvert \leq  2\tau(L) - 2 \leq 2g_4(L)$$ for any Legendrian diagram $\mathcal{L}$ for a 2--component link $L$. If we apply this inequality to $\ell_i$, using Proposition~\ref{prop:legformulas} and the diagram in Figure~\ref{fig:Li}, we get the following: $$ 2i + |0| \leq  2\tau(\ell_i) - 2 \leq 2i.$$ Then we see that $\tau(\ell_i)=i+1$ and the inequality is sharp for $\ell_i$. This establishes the following corollary.

\begin{corollary}Cavallo's $\tau$--invariant can be arbitrarily large for non-split topologically slice 2--component links. 
\end{corollary}

\begin{remark}An anonymous referee suggested the following slightly different approach to the proof of the main theorem of this paper. Let $J$ be the positive untwisted Whitehead double of the right handed trefoil. Start with the (2,0) cable of $J$, with antiparallel strands, and performing a connect-sum locally with $\#_n J$. As in our proof, we can find cobordisms to knots with known slice genera to conclude that the slice genus of the link is $n$. These links also satisfy the requirements of Theorem~\ref{thm_main}.
\end{remark}
%=============================================================

\bibliographystyle{alpha}
\bibliography{knotbib}

\end{document}